\documentclass[11pt]{article}
\usepackage{amsmath,amsthm,amsfonts,amssymb}
\usepackage{subfigure}
\usepackage{epsfig}
\usepackage[usenames]{color}
\usepackage{verbatim}

\parskip 	\bigskipamount

\newtheorem{theorem}{Theorem}[section]

\numberwithin{equation}{section}
\newtheorem{lemma}[theorem]{Lemma}
\newtheorem{proposition}[theorem]{Proposition}
\newtheorem{remark}[theorem]{Remark}

\numberwithin{equation}{section}

\def\N{\mathbb{N}}

\def\R{\mathbb{R}}
\def\C{\mathbb{C}}
\def\P{\mathbb{P}}
\def\E{\mathbb{E}}

\def\S{\mathbb{S}}

\def\B{\mathcal{B}}

\renewcommand{\phi}{\varphi}
\renewcommand{\epsilon}{\varepsilon}

\newcommand{\1}{{\text{\Large $\mathfrak 1$}}}

\newcommand{\vol}{\mathrm{vol}}

\def\C{{\mathcal C}}
\newcommand{\tdet}{T}

\newcommand\be{\begin{equation}}
\newcommand\ee{\end{equation}}

\begin{document}

\title{\bf An isoperimetric inequality for the Wiener sausage}

\author{
Yuval Peres\thanks{Microsoft Research, Redmond, Washington, USA; peres@microsoft.com} \and Perla Sousi\thanks{University of Cambridge, Cambridge, UK;   p.sousi@statslab.cam.ac.uk}
}
\maketitle
\begin{abstract}
Let $(\xi(s))_{s\geq 0}$ be a standard Brownian motion in $d\geq 1$ dimensions and let $(D_s)_{s \geq 0}$ be a collection of open sets in $\R^d$. For each $s$, let $B_s$ be a ball centered at $0$ with $\vol(B_s) = \vol(D_s)$.
We show that $\E[\vol(\cup_{s \leq t}(\xi(s) + D_s))] \geq \E[\vol(\cup_{s \leq t}(\xi(s) + B_s))]$, for all $t$.
In particular, this implies that the expected volume of the Wiener sausage increases when a drift is added to the Brownian motion.
\newline
\newline
\emph{Keywords and phrases.} Poisson point process, random walk, Brownian motion, coupling, rearrangement inequality.
\newline
MSC 2010 \emph{subject classifications.}
Primary   26D15, 
          60J65, 
Secondary 60D05, 
          60G55, 
          60G50. 
\end{abstract}

\section{Introduction}
Let $(\xi(t))_{t \geq 0}$ be a standard Brownian motion in $\R^d$. For any open set $A \subset \R^d$, the corresponding Wiener sausage at time $t$ is $\cup_{s \leq t}(\xi(s) + A)$.

Kesten, Spitzer and Whitman(1964) (see \cite[p.\ 252]{ItoMckean}) proved that for $d \geq 3$, if $A \subset \R^d$ is an open set with finite volume, then the Wiener sausage satisfies
\[
\frac{\E[\vol(\cup_{s \leq t} (\xi(s)+ A))]}{t} \to \text{Cap}(A) \text{ as } t \to \infty,
\]
where $\text{Cap}(A)$ is the Newtonian capacity of the set $A$.

P\'olya and Sz\"ego (see \cite{PoSz}) proved that for $d \geq 3$, among all open sets of fixed volume, the ball has the smallest Newtonian capacity.
Thus, for any open set $A \subset \R^d$ of finite volume,
\begin{align}\label{eq:polyaszego}
\E [\vol(\cup_{s \leq t}(\xi(s) + A))] \geq (1 - o(1)) \E[\vol(\cup_{s \leq t}(\xi(s) + B))], \text{ as } t \to \infty,
\end{align}
where $B$ is a ball with $\vol(B) = \vol(A)$.

This naturally raises the question whether \eqref{eq:polyaszego} holds for fixed $t$ without the $1 - o(1)$ factor. Our main result gives a positive answer in the more general setting where the set $A$ is allowed to vary with time.

\begin{theorem}\label{thm:wienersausage}
Let $(\xi(s))_{s\geq 0}$ be a standard Brownian motion in $d\geq 1$ dimensions and let $(D_s)_{s \geq 0}$ be open sets in $\R^d$. For each $s$, let $r_s>0$ be such that $\vol(\B(0,r_s)) = \vol(D_s)$. Then for all $t$ we have that
\[
\E\left[\vol\left(\cup_{s \leq t}\left( \xi(s) + D_s  \right)\right)\right] \geq \E\left[\vol\left(\cup_{s \leq t}\B(\xi(s),r_s)\right)\right].
\]
\end{theorem}

Our original motivation for Theorem~\ref{thm:wienersausage} came from our joint work with A.~Sinclair and A.~Stauffer \cite{PSSS10}. In \cite{PSSS10} it was proved that in dimension $2$ for any continuous function $f:\R_+ \to \R^2$,
\[
\E[\vol(\cup_{s \leq t}\B(\xi(s)+f(s),r))] \geq (1-o(1))\E[\vol(\cup_{s \leq t}\B(\xi(s),r))], \text{ as } t \to \infty
\]
and it was conjectured that for all $d \geq 1$ and for any continuous function $f:\R_+ \to \R^d$,
\begin{align}\label{eq:drift}
\E[\vol(\cup_{s \leq t}\B(\xi(s)+f(s),r))] \geq \E[\vol(\cup_{s \leq t}\B(\xi(s),r))],
\end{align}
where $\B(x,r)$ stands for the open ball centered at $x$ of radius $r$.

In dimension $1$ this conjecture was shown in \cite{PSSS10} to follow from the reflection principle. Theorem~\eqref{thm:wienersausage} above establishes \eqref{eq:drift} without requiring $f$ to be continuous or even measurable, by taking $D_s = \B(f(s),r)$.

\begin{figure}
\begin{center}
\subfigure[Wiener sausage with squares]{
\epsfig{file=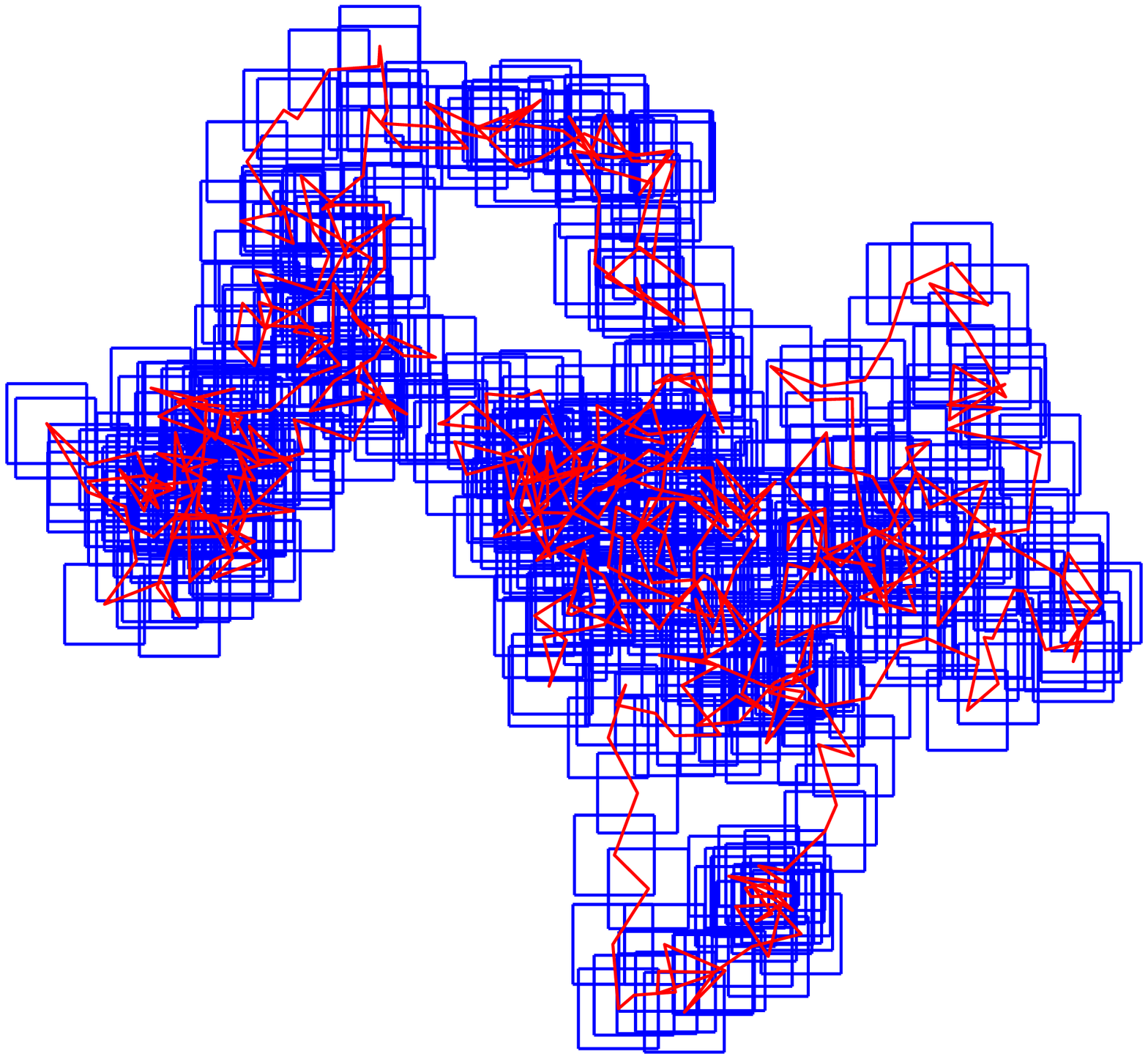,height=6cm}
}
\subfigure[Wiener sausage with discs]{
\epsfig{file=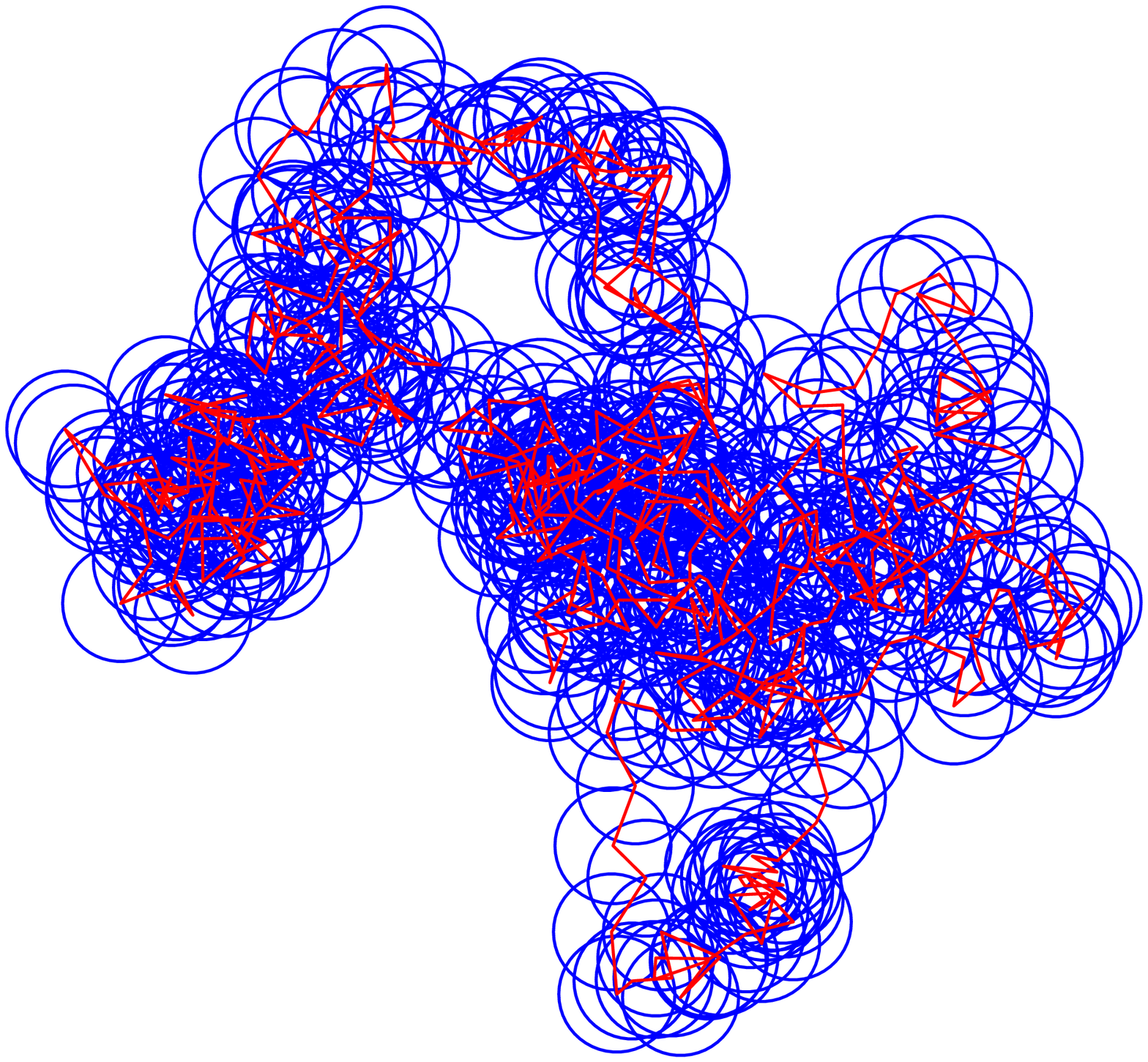,height=6cm}
}
\end{center}
\end{figure}

\begin{remark}
\rm{
For any collection of open sets $(D_s)$ the set $\cup_{s \leq t}(\xi(s) + D_s)$ is open, and hence Lebesgue measurable. Its volume, namely $\vol(\cup_{s \leq t} (\xi(s) + D_s))$, is a random variable. This is explained at the end of the proof of Theorem~\ref{thm:wienersausage} in Remark~\ref{rem:measure}.
}
\end{remark}

\begin{remark}
\rm{
The Wiener sausage determined by a ball, i.e.\ $W(t) = \cup_{s \leq t}\B(\xi(s),r)$, has been studied extensively. Spitzer \cite{Spitzer} obtained exact asymptotics as $t \to \infty$ for the expected volume of the Wiener sausage in $2$ and $3$ dimensions
and Donsker and Varadhan \cite{DonVar} obtained exact asymptotics for the exponential moments of the volume. Kesten, Spitzer and Whitman (see \cite{Spitzer} and \cite{Whitman}) proved laws of large numbers type results and
Le Gall in \cite{LeGall} proved fluctuation results corresponding to these laws of large numbers.
}
\end{remark}

\begin{remark}
\rm{
In \cite{MOBC} (see also \cite[Corollary~2.1]{DGRS}) a result analogous to \eqref{eq:drift} for random walks is proved, namely that the expected range of a lattice walk always increases when a drift is added to the walk. The proof of that result and our proof of Theorem~\ref{thm:wienersausage} do not seem to yield each other.
}
\end{remark}

\begin{remark}
\rm{
A result analogous to \eqref{eq:drift} was proved in \cite{PS10} for the Hausdorff dimension of the image and the graph of $\xi+f$, where $f$ is a continuous function. Namely, in \cite{PS10} it is proved that
\[
\dim (\xi+f)[0,1] \geq \max\{\dim \xi[0,1], \dim f[0,1]\} \text{ a.s.}
\]
and similarly for the dimension of the graph.
}
\end{remark}

To prove Theorem~\ref{thm:wienersausage}, we establish an analogous statement for certain random walks, which is of independent interest:
\begin{proposition}\label{lem:discrete}
Let $\epsilon>0$ and let $(z(k))_{k\geq 0}$ be a random walk in $\R^d$ with transition kernel given by
\begin{align}\label{eq:kernel}
p(x,y) = \frac{\1(\|x-y\|<\epsilon)}{\epsilon^d \omega(d)},
\end{align}
where $\omega(d)$ is the volume of the unit ball in $\R^d$.
For any integer $n$ and any collection of open sets $(U_k)_{k\geq 0}$ in $\R^d$ we have that
\begin{align}\label{eq:sausagerw}
\E\left[\vol \left(\cup_{k=0}^{n} \left(z(k) + U_k\right)\right)\right]
\geq \E\left[\vol \left(\cup_{k=0}^{n} \B(z(k),r_k)\right)\right],
\end{align}
where $r_k$ is such that $\vol(\B(0,r_k)) = \vol(U_k)$.
\end{proposition}

In the next section we first state a rearrangement inequality in Theorem~\ref{thm:rearrang}, taken from \cite[Theorem~2]{BurSch}, and apply it to random walks on the sphere. Then we prove Proposition~\ref{lem:discrete}.
\newline
In Section~\ref{sec:actualproof}, using Proposition~\ref{lem:discrete} and Donsker's invariance principle, we give the proof of Theorem~\ref{thm:wienersausage}.
\newline
In Section~\ref{sec:convergence} we collect some easy convergence lemmas.
\newline
Finally, in Section~\ref{sec:conclusion} we conclude with some open questions and remarks.

\section{Rearrangement inequalities and proof of Proposition~\ref{lem:discrete}}\label{sec:proof}

Let $\S$ denote a sphere in $d$ dimensions. We fix $x^* \in \S$. For a subset $A$ of $\S$ we define $A^*$ to be a geodesic cap centered at $x^*$ such that $\mu(A^*) = \mu(A)$, where $\mu$ is the surface area measure on the sphere. We call $A^*$ the symmetric rearrangement of $A$.

The following theorem is a special case of Burchard and Schmuckenschl\"ager \cite[Theorem~2]{BurSch}.
\begin{theorem}[\cite{BurSch}]\label{thm:rearrang}
Let $A_1,\ldots,A_n$ be subsets of\, $\S$ and let $\psi_{ij}:\S\times \S \to \R_+$ be nonincreasing functions of distance.
Then we have the following inequality
\begin{align*}
\int_{\S}\ldots \int_{\S} \prod_{1 \leq i \leq n} \1(x_i \in A_i) &\prod_{1\leq i <j\leq n}\psi_{ij}(x_i,x_j)\,d\mu(x_1)\ldots\,d\mu(x_n) \\
& \leq \int_{\S}\ldots \int_{\S} \prod_{1 \leq i \leq n} \1(x_i \in A_i^*) \prod_{1\leq i <j \leq n}\psi_{ij}(x_i,x_j)\,d\mu(x_1)\ldots\,d\mu(x_n).
\end{align*}
\end{theorem}

Let $\S_R \subset \R^{d+1}$ be the sphere of radius $R$ in $d+1$ dimensions centered at $0$, i.e.\
\begin{align}\label{eq:defsph}
\S_R = \{(x_1, \ldots, x_{d+1}) \in \R^{d+1}: \sum_{i=1}^{d+1} x_i^2  = R^2 \}.
\end{align}
Let $\epsilon>0$ and let $\tilde{\zeta}$ be a random walk on the sphere that starts from a uniform point on the surface of the sphere and has transition kernel given by
\begin{align}\label{eq:sphkern}
\psi(x,y) = \frac{\1(\rho(x,y)<\epsilon)}{\mu(\C(x,\epsilon))}.
\end{align}
Again $\mu$ stands for the surface area measure on the sphere, $\rho(x,y)$ stands for the geodesic distance between $x$ and $y$ and
\[
\C(x,\epsilon) = \{z \in \S_R: \rho(x,z)<\epsilon \}.
\]
For a collection $(\Theta_k)_{k\geq 0}$ of subsets of $\S_R$ we define
\[
\tau^{\Theta}= \inf\{k \geq 0:  \tilde{\zeta}(k) \in \Theta_k \}.
\]

\begin{lemma}\label{lem:detsphere}
Let $(\Theta_k)_{k\geq 0}$ be subsets of the sphere, $\S_R$. Then for all $n$ we have that 
\begin{align}
\P(\tau^{\Theta}>n) \leq \P(\tau^{\C}>n),
\end{align}
where for each $k$ we define $\C_k$ to be a geodesic cap centered at $(0,\ldots,0,-R)$ such that
$\mu(\C_k) = \mu(\Theta_k)$.
\end{lemma}

\begin{proof}[\textbf{Proof}]
Using the Markov property, we write $\P(\tau^{\Theta} >n)$ as
\begin{align*}
\frac{1}{\mu(\S)}
\int_{\S_R}\cdots \int_{\S_R} \prod_{i=1}^{n} \psi(x_{i-1},x_i) \prod_{i=0}^{n}\1(x_{i} \in \Theta_i^{\mathrm{c}})\,d\mu(x_0)\cdots \,d\mu(x_n).
\end{align*}
Let $x^* = (0,\ldots,0,R)$. Since the transition kernel $\psi$ is a nonincreasing function of distance, applying Theorem~\ref{thm:rearrang}, we obtain that this last integral is bounded from above by
\[
\int_{\S_R} \ldots \int_{\S_R} \prod_{i=1}^{n} \psi(x_{i-1},x_i) \prod_{i=0}^{n}\1(x_{i} \in A_i)\,d\mu(x_0)\cdots \,d\mu(x_n),
\]
where for each $i$, we have that $A_i$ is a geodesic cap centered at $x^*$ with $\mu(A_i) = \mu(\Theta_i^{\mathrm{c}})$.
Writing $\C_i = A_i^{\mathrm{c}}$, we have that $\C_i$ is a geodesic cap centered at $(0,\ldots,0,-R)$ such that
$\mu(\C_i) = \mu(\Theta_i)$ and hence proving the lemma.
\end{proof}

The rest of this section is devoted to the proof of Proposition~\ref{lem:discrete}.

\begin{proof}[\textbf{Proof of Proposition~\ref{lem:discrete}}]
First note that if the sets $(U_k)$ have infinite volume, then the inequality is trivially true.
We will prove the theorem under the assumption that they are all bounded, since then by truncation we can get the result for any collection of unbounded open sets.

Let $L>0$ be sufficiently large such that
\[
\cup_{i=0}^{n}U_i \subset \B(0,L).
\]
We define the projection mapping $\pi: \S_R \to \R^d$ via
\begin{align}\label{eq:defproj}
\pi(x_1, \ldots, x_{d+1}) = (x_1, \ldots, x_d)
\end{align}
and the inverse map $\pi^{-1}(x_1,\ldots,x_d) = (x_1,\ldots,x_{d+1}) \in \S_R$ such that $x_{d+1}<0$.
Next, we consider the ball $\B(0,L+cn\epsilon)$ in $\R^d$, where $c$ is a positive constant chosen in such a way so that if we also take $R$ big enough, then  we can ensure that
\begin{align}\label{eq:defcap}
\C(L) = \pi^{-1}(\B(0,L+cn\epsilon))
\end{align}
is a cap (centered at $(0,\ldots,0,-R)$) of geodesic radius bigger than $L+n \epsilon$.
\newline
From Lemma~\ref{lem:detsphere} we have that
\begin{align}\label{eq:ineqsphe}
\P(\tau^{\pi^{-1}(U)}>n) \leq \P(\tau^{\C}>n),
\end{align}
where for each $k$ we define $\C_k$ to be a geodesic cap centered at $(0,\ldots,0,-R)$ such that
$\mu(\C_k) = \mu(\pi^{-1}(U_k))$.
\newline
It is clear that
\[
\P(\tau^{\pi^{-1}(U)} \leq n) = \P(\tilde{\zeta}(0) \in \C(L)) \P_{\C(L)}(\tau^{\pi^{-1}(U)} \leq n),
\]
where under $\P_{\C(L)}$ the starting point of the random walk is uniform on $\C(L)$. Similarly for the collection of caps $\C$ we have
\[
\P(\tau^{\C} \leq n) = \P(\tilde{\zeta}(0) \in \C(L)) \P_{\C(L)}(\tau^{\C} \leq n).
\]
Hence, using the above equalities together with \eqref{eq:ineqsphe} we obtain that
\begin{align}\label{eq:finineqsph}
\P(\forall k=0,\ldots,n, \pi(\zeta(k)) \notin U_k) \leq \P(\forall k=0,\ldots,n, \pi(\zeta(k)) \notin \pi(\C_k)),
\end{align}
where $\zeta$ is a random walk started from a uniform point in $\C(L)$ and with transition kernel given by \eqref{eq:sphkern}.
Now, we will go back to $\R^d$.
Let $(z(k))_k$ be a random walk that starts from a uniform point in $\B(0,L+cn\epsilon)$ and has transition kernel given by \eqref{eq:kernel}. We define
\begin{align}\label{eq:deftdet}
\tdet^{U} = \min \{k \geq 0: z(k) \in U_k \}.
\end{align}
Since $z(0)$ is uniform on $\B(0,L+cn\epsilon)$, we can write
\[
\P(\tdet^{U} > n) = 1 - \frac{\E\left[\vol\left(\cup_{k=0}^{n}(z(k)+U_k) \right)\right]}{\vol(\B(0,L+cn\epsilon))}.
\]
It is easy to check that for each $k$ the projection $\pi(\C_k)$ is a ball in $\R^d$ centered at $0$. Let $r_k^R$ be its radius and let
$r_k$ be such that $\vol(\B(0,r_k)) = \vol(U_k)$.
Then by Lemma~\ref{lem:volume} we have that
\begin{align}\label{eq:radiusconv}
r_k^R  \to r_k, \text{ as } R \to \infty.
\end{align}
From Lemma~\ref{lem:coupling}, there exists a coupling of $(\pi(\zeta(k)))_{k=0,\ldots,n}$ with $(z(k))_{k=0,\ldots,n}$, so that
\begin{align*}
\P(\forall k=0,\ldots,n: \pi(\zeta(k)) = z(k)) \to 1 \text{ as } R \to \infty.
\end{align*}
By the union bound we have that
\begin{align}\label{eq:equal1}
\P(\forall k=0,\ldots,n: z(k) \notin U_k) \leq \P(\forall k=0,\ldots,n: \pi(\zeta(k)) \notin U_k) + \P(\text{coupling fails}),
\end{align}
where
\begin{align}\label{eq:failure}
\P(\text{coupling fails}) = 1 - \P(\forall k=0,\ldots,n: \pi(\zeta(k)) = z(k)) \to 0 \text{ as } R \to \infty.
\end{align}
Similarly,
\begin{align}\label{eq:equal2}
\P(\forall k=0,\ldots,n: z(k) \notin \pi(\C_k)) \geq \P(\forall k=0,\ldots,n: \pi(\zeta(k)) \notin \pi(\C_k)) - \P(\text{coupling fails}).
\end{align}
Hence, \eqref{eq:equal1} and \eqref{eq:equal2} together with \eqref{eq:finineqsph} give that
\begin{align*}
\P(\tdet^{\pi(\C)} >n) \geq \P(\forall k, z(k) \notin U_k) - 2\P(\text{coupling fails}).
\end{align*}
Therefore, using that $\pi(\C_k) = \B(0,r_k^R)$, we deduce that
\begin{align*}
1 - \frac{\E\left[\vol\left(\cup_{k=0}^{n}(z(k)+\B(0,r_k^R)) \right)\right]}{\vol(\B(0,L+cn\epsilon))}
\geq 1 -\frac{\E\left[\vol\left(\cup_{k=0}^{n}(z(k)+U_k) \right)\right]}{\vol(\B(0,L+cn\epsilon))} - 2\P(\text{coupling fails}).
\end{align*}
Using \eqref{eq:radiusconv} and \eqref{eq:failure} and letting $R \to \infty$ in the inequality above concludes the proof of the lemma.
\end{proof}

\section{Proof of Theorem~\ref{thm:wienersausage}}\label{sec:actualproof}
For $n \in \N$ and $t >0$ we define
\[
\mathcal{D}_{n,t}  = \left\{\frac{kt}{2^n}: k=0,\ldots,2^n \right\}.
\]

\begin{lemma}\label{lem:browniandiscrete}
Let $(\xi(s))_{s \geq 0}$ be a standard Brownian motion in $d \geq 1$ dimensions and let $(V_s)_{s \geq 0}$ be open sets in $\R^d$. For each $s$, let $r_s>0$ be such that $\vol(V_s) = \vol(\B(0,r_s))$.
Then for all $n \in \N$ and $t>0$ we have that
\[
\E\left[\vol \left(\cup_{\ell \in \mathcal{D}_{n,t}} (\xi(\ell) + V_\ell)\right) \right] \geq
\E\left[ \vol \left(\cup_{\ell \in \mathcal{D}_{n,t}} \B(\xi(\ell), r_\ell) \right) \right].
\]
\end{lemma}

\begin{proof}[\textbf{Proof}]
First note that if the sets $(V_s)$ have infinite volume, then the inequality is trivially true.
We will prove the theorem under the assumption that they are all bounded, since then by truncation we can get the result for any collection of unbounded open sets.

Let $X_1,X_2,\ldots$ be i.i.d.\ random variables with uniform distribution on $\B(0,1)$. Write $S_n = \sum_{i=1}^{n}X_i$. Let $N \in \N$ and let $\sigma^2$ be the variance of $X_1$. Then for any collection of bounded open sets $(U_s)$ we have from Proposition~\ref{lem:discrete} that
\begin{align}\label{eq:expect}
\E\left[\vol\left(\cup_{s \in \{0,1/N,\ldots,[t]/N \}}\left(\frac{S_{[Ns]}}{\sqrt{\sigma^2 N}} +U_s \right) \right)\right]
\geq \E\left[\vol\left(\cup_{s \in \{0,1/N,\ldots,[t]/N \}}\left(\frac{S_{[Ns]}}{\sqrt{\sigma^2 N}} + \B(0,r_s) \right) \right)\right],
\end{align}
where $r_s$ is such that $\vol(U_s) = \vol(\B(0,r_s))$.
\newline
We will drop the dependence on $t$ from $\mathcal{D}_{n,t}$ to simplify the notation.
Since \eqref{eq:expect} holds true for any collection of  bounded sets, we have that for $N$ large enough
\begin{align*}
\E\left[\vol\left(\cup_{\ell \in \mathcal{D}_n}\left(\frac{S_{[N\ell]}}{\sqrt{\sigma^2 N}} +V_\ell \right) \right)\right]
\geq \E\left[\vol\left(\cup_{\ell \in \mathcal{D}_n}\left(\frac{S_{[N\ell]}}{\sqrt{\sigma^2 N}} + \B(0,r_\ell) \right) \right)\right],
\end{align*}
since in \eqref{eq:expect} we can take $U_\ell=V_\ell$, for $\ell \in \mathcal{D}_n$ and empty otherwise. As before, $r_\ell$ is such that $\vol(\B(0,r_\ell)) = \vol(V_\ell)$.
\newline
By Donsker's invariance principle (see for instance \cite[Theorem~5.22]{BM}), for a fixed $n$, we have that
\begin{align}\label{eq:donsker}
\left(\frac{S_{[N\ell]}}{\sqrt{\sigma^2 N}}\right)_{\ell \in \mathcal{D}_n} \xrightarrow{\text{(\emph{w})}} (\xi(\ell))_{\ell \in \mathcal{D}_n}
\text{ as } N \to \infty.
\end{align}
If for all $\ell \in \mathcal{D}_n$ we have that $\vol(\partial V_\ell) = 0$, then by \eqref{eq:donsker} and Lemma~\ref{lem:convol} we deduce that
\[
\E\left[\vol\left(\cup_{\ell \in \mathcal{D}_n}\left(\xi(\ell) +V_\ell \right) \right)\right]
\geq \E\left[\vol\left(\cup_{\ell \in \mathcal{D}_n}\left(\xi(\ell) + \B(0,r_\ell) \right) \right)\right].
\]
If $\exists k$ such that $\vol(\partial V_k) >0$, then for each $\ell \in \mathcal{D}_n$ we write $V_\ell = \cup_{j=1}^{\infty}A_{j,\ell}$, where $(A_{j,\ell})_j$ are all the dyadic cubes that are contained in $V_\ell$.
\newline
Then for every finite $K$, we have that
\begin{align}\label{eq:cube}
\E\left[\vol\left(\cup_{\ell \in \mathcal{D}_n}\left(\xi(\ell) +\cup_{j=1}^{K}A_{j,\ell} \right) \right)\right]
\geq \E\left[\vol\left(\cup_{\ell \in \mathcal{D}_n}\left(\xi(\ell) + \B(0,r(\ell,K)) \right) \right)\right],
\end{align}
where for each $\ell$ we have that $r(\ell,K)$ satisfies $\vol(\B(0,r(\ell,K))) = \vol(\cup_{j=1}^{K}A_{j,\ell})$, and hence $r(\ell,K) \nearrow r_\ell$ as $K \to \infty$. Thus, letting $K \to \infty$, by monotone convergence we conclude that
\[
\E\left[\vol\left(\cup_{\ell \in \mathcal{D}_n}\left(\xi(\ell) + V_\ell \right) \right)\right]
\geq \E\left[\vol\left(\cup_{\ell \in \mathcal{D}_n}\left(\xi(\ell) + \B(0,r_\ell) \right) \right)\right].
\]
\end{proof}

\begin{proof}[\textbf{Proof of Theorem~\ref{thm:wienersausage}}]

Let $M>0$. We will show the theorem under the assumption that $D_s \subset \B(0,M)$, for all $s \leq t$, since then by monotone convergence we can get it for any collection of open sets.
\newline
For each $s$, we define the set
\begin{align*}
D_{s,n} = \{z \in D_s: d(z,D_s^{\mathrm{c}}) > \left(t/2^n\right)^{1/3}\},
\end{align*}
which is open, since the sets $(D_s)$ were assumed to be open.
For every $\ell \in \mathcal{D}_n$, we define
\begin{align}\label{eq:defz}
Z_{\ell} = \cup_{\ell \leq s < \ell + t/2^n} D_{s,n},
\end{align}
which is again open as a union of open sets.
For each $s$ we let $r_{s,n}$ be such that $\vol(\B(0,r_{s,n})) = \vol(D_{s,n})$.
From Proposition~\ref{lem:discrete} we then get that
\begin{align}\label{eq:finalineq}
\E\left[\vol\left(\cup_{\ell \in \mathcal{D}_n}\left(\xi(\ell) +Z_{\ell} \right) \right)\right]
\geq \E\left[\vol\left(\cup_{\ell \in \mathcal{D}_n}\left(\xi(\ell) + \B(0,r_{\ell,n}^*) \right) \right)\right],
\end{align}
where $r_{\ell,n}^*$ satisfies
\[
r_{\ell,n}^* = \left( \sup_{\ell \leq s < \ell + t/2^n} r_{s,n}  - (t/2^n)^{1/3} \right)^+.
\]
We now define the event
\[
\Omega_n = \left\{ \forall h \leq t/2^n: \sup_{s,u: |s - u|\leq h} \|\xi(s) - \xi(u) \| \leq (2^{1/3} - 1)h^{1/3}\right \}.
\]
We will now show that on $\Omega_n$ we have that
\[
\cup_{\ell \in \mathcal{D}_n}(\xi(\ell) + Z_\ell) \subset \cup_{s \leq t} (\xi(s) + D_s).
\]
Let $y \in \xi(\ell)+ Z_\ell$, for some $\ell \in \mathcal{D}_n$. Then
there exists $\ell \leq s < \ell + t/2^n$ such that $d(y-\xi(\ell),D_s^{\mathrm{c}}) > (t/2^n)^{1/3}$. Also, since $\| \xi(s) - \xi(\ell)\| \leq (t/2^n)^{1/3}$ on $\Omega_n$, we get by the triangle inequality that
\[
d(y-\xi(s),D_s^{\mathrm{c}}) \geq d(y-\xi(\ell),D_s^{\mathrm{c}}) - \|\xi(s) - \xi(\ell)\| > (t/2^n)^{1/3} - (2^{1/3} - 1)(t/2^n)^{1/3} >0,
\]
and hence we deduce that $y \in \xi(s) + D_s$.
\newline
Therefore, from \eqref{eq:finalineq} we obtain
\begin{align}\label{eq:beforefinal}
\E\left[\vol \left( \cup_{s \leq t} (\xi(s) + D_s) \right)\right] + \E\left[\vol \left( \cup_{\ell \in \mathcal{D}_n} (\xi(\ell) + Z_\ell) \right) \1(\Omega_n^{\mathrm{c}})\right] \geq
\E\left[\vol \left( \cup_{\ell \in \mathcal{D}_n} \B(\xi(\ell),r_{\ell,n}^*) \right)\1(\Omega_n)\right].
\end{align}
For all $n$ we have that a.s.\
\[
\vol \left( \cup_{\ell \in \mathcal{D}_n} \B(\xi(\ell),r_{\ell,n}^*) \right)\1(\Omega_n) \nearrow \vol\left(\cup_{s \leq t} \B(\xi(s),r_s) \right).
\]
Indeed it is clear that $\Omega_n \subset \Omega_{n+1}$ and on $\Omega_n$ we have that
\[
\cup_{\ell \in \mathcal{D}_n}\B(\xi(\ell),r_{\ell,n}^*) \subset \cup_{\ell \in \mathcal{D}_{n+1}} \B(\xi(\ell),r_{\ell,n+1}^*).
\]
To prove that, let $y$ be such that $\|y- \xi(\ell) \| < r_{\ell,n}^*$. We set $\ell' = \ell + t/2^{n+1}$ and $\ell''=\ell + t/2^n$. Then there are two cases:
\begin{itemize}
\item If $\sup_{\ell \leq s < \ell'} r_{s,n} \geq \sup_{\ell' \leq s < \ell''}r_{s,n}$, then
$\| y - \xi(\ell)\| < r_{\ell,n+1}^*$.
\item If $\sup_{\ell \leq s < \ell'} r_{s,n} < \sup_{\ell' \leq s < \ell''}r_{s,n}$, then
\begin{align*}
\| y - \xi(\ell') \| \leq \| y - \xi(\ell) \| + \| \xi(\ell) - \xi(\ell') \| < r_{\ell,n}^* + (2^{1/3} - 1) (t/2^{n+1})^{1/3} \\
= \sup_{\ell' \leq s < \ell+t/2^n} r_{s,n} - (t/2^n)^{1/3} + (t/2^{n+1})^{1/3}(2^{1/3} - 1) \leq r_{\ell',n+1}^*.
\end{align*}
\end{itemize}
It is easy to show that on $\cup_n \Omega_n$ we have $\cup_{n} \cup_{\ell \in \mathcal{D}_n}\B(\xi(\ell),r_{\ell,n}^*) = \cup_{s \leq t} \B(\xi(s),r_s)$.
\newline
Finally by the dominated convergence theorem we get that
\[
\E\left[\vol \left( \cup_{\ell \in \mathcal{D}_n} (\xi(\ell) + Z_\ell) \right) \1(\Omega_n^{\mathrm{c}})\right] \to 0,
\]
since $\P(\Omega_n^{\mathrm{c}}) \to 0$ as $n \to \infty$, by L\'evy's modulus of continuity theorem (see for instance \cite[Theorem~1.14]{BM}) and also
\[
\vol\left(\cup_{\ell \in \mathcal{D}_n}\left(\xi_\ell + Z_\ell \right)\right) \leq \vol(\B(0,\max_{s \leq t}|\xi(s)|+M)),
\]
since we assumed that the sets $(D_s)$ are all contained in $\B(0,M)$.
\newline
Taking the limit as $n \to \infty$ in \eqref{eq:beforefinal} concludes the proof.
\end{proof}

\begin{remark}\label{rem:measure}
\rm{
The proofs of Lemma~\ref{lem:browniandiscrete} and Theorem~\ref{thm:wienersausage} also give that for any collection of open sets $(D_s)$, the volume of $\cup_{s \leq t}(\xi(s)+D_s)$ is a random variable. Indeed, on $\cup_{n}\Omega_n$ (which has $\P(\cup_n \Omega_n) = 1$) we have that
\[
\vol(\cup_{s \leq t}(\xi(s) + D_s)) = \lim_{n \to \infty}\vol\left(\cup_{k=1}^{n}\cup_{\ell \in \mathcal{D}_k} (\xi(\ell) + Z_\ell)\right),
\]
where $Z_\ell$ is as defined in \eqref{eq:defz}. If for all $n$ and all $\ell \in \mathcal{D}_n$
we have that $\vol(\partial Z_\ell) = 0$, then from Lemma~\ref{lem:convol} we get the measurability. Otherwise, we write $Z_\ell$ as a countable union of the dyadic subcubes contained in it and then use the monotonicity property together with Lemma~\ref{lem:convol} again, like in the last part of the proof of Lemma~\ref{lem:browniandiscrete}.
}
\end{remark}

\section{Relating the sphere to its tangent plane}\label{sec:convergence}

It is intuitive and standard that as $R \to \infty$ the sphere $\S_R$ tends to Euclidean space. This section, that can be skipped, makes this precise in the situation we need and also establishes a useful lemma on continuity of volumes.

As in Section~\ref{sec:proof}, $\C(x,r)$ stands for the geodesic cap centered at $x$ of radius $r$. Also, the projection mapping $\pi$ is as defined in \eqref{eq:defproj} and $\C(K)=\pi^{-1}(\B(0,K))$, for $K>0$, is a cap centered at $(0,\ldots,0,-R)$, where $\B(0,K)$ is a ball in $\R^d$ centered at $0$ of radius $K$.
\begin{lemma}\label{lem:volume}
Let $r>0$. Then for all $\delta>0$, there exists $R_0$ such that for all $x \in \B(0,K)$ we have that
\begin{align}\label{eq:proj}
\B(x,r-\delta) \subset \pi(\C(\pi^{-1}(x), r)) \subset \B(x,r+\delta), \text{ for all } R>R_0.
\end{align}
Also for all $A \subset \pi^{-1}(\B(0,K))$ we have that
\[
\mu(A) - \vol(\pi(A)) \to 0, \text{ as } R \to \infty.
\]
\end{lemma}
\begin{proof}[\textbf{Proof}]
Let $c$ be a sufficiently large positive constant so that for $R$ sufficiently large we have for all $x \in \B(0,K)$ that
\[
\pi(\C(\pi^{-1}(x),r)), \B(x,r+\delta) \subset \B(0,K+c).
\]
Let $x= (x_1,\ldots,x_d) \in \B(0,K)$ and $y=(y_1,\ldots,y_d) \in \pi(\C(\pi^{-1}(x), r))$.
\newline Then $\pi^{-1}(y) \in \C(\pi^{-1}(x),r)$ and hence
\begin{align}\label{eq:deftheta}
\theta(\pi^{-1}(x),\pi^{-1}(y)) \leq \frac{r}{R},
\end{align}
where $\theta=\theta(\pi^{-1}(x),\pi^{-1}(y))$ satisfies
\[
\cos \theta = \frac{\langle \pi^{-1}(x),\pi^{-1}(y) \rangle}{R^2} = \frac{\sum_{i=1}^{d+1}x_i y_i}{R^2}.
\]
We write $\|\cdot\|$ for the Euclidean distance in $\R^d$. We then have
\begin{align*}
\|x-y\|^2 = \sum_{i=1}^{d}(x_i - y_i)^2 = 2R^2 - 2 R^2 \cos \theta - (x_{d+1} - y_{d+1})^2.
\end{align*}
But, since $\pi^{-1}(x), \pi^{-1}(y) \in \S_R$, we get that $x_{d+1} = \sqrt{R^2 - \sum_{i=1}^{d}x_i^2}$ and
$y_{d+1} = \sqrt{R^2 - \sum_{i=1}^{d}y_i^2}$ and using the fact that
$x \in \B(0,K)$ and $y\in \B(0,K+c)$ we obtain that
\[
(x_{d+1} - y_{d+1}) \to 0 \text{ as } R \to \infty,
\]
uniformly over all $x$ and $y$ in $\B(0,K)$ and $\B(0,K+c)$ respectively. From \eqref{eq:deftheta} and the monotonicity of the $\cos \phi$ function for $\phi$ small, we obtain that
\[
\cos \theta \geq \cos \frac{r}{R},
\]
and hence for all $R$ sufficiently large we have that
\[
\|x-y\| \leq r + \delta.
\]
To prove the other inclusion, it suffices to show that
\begin{align}\label{eq:inclusion}
\B(0,K+c) \cap \pi(\C(\pi^{-1}(x),r))^{\mathrm{c}} \subset \B(0,K+c) \cap \B(x,r-\delta)^{\mathrm{c}},
\end{align}
since $\B(0,K+c)^{\mathrm{c}}$ is disjoint from $\pi(\C(\pi^{-1}(x),r))$ and $\B(x,r-\delta)$. Finally
\eqref{eq:inclusion} can be proved using similar arguments to the ones employed above.

To prove the last statement of the lemma, we write
\[
\mu(A) - \vol(\pi(A)) = \int_{\pi(A)} \left(\frac{1}{\cos(\alpha(x))} - 1\right)\,dx,
\]
where $\alpha(x)$ is the angle at the origin between $\pi^{-1}(x)$ and $(0,\ldots,0,-R)$.
\newline
Using that $\alpha(x) \to 0$ as $R \to \infty$ uniformly over all $x \in \B(0,K)$ we get the desired convergence.
\end{proof}

\begin{lemma}\label{lem:coupling}
Let $z$ be a random walk in $\R^d$ started from a uniform point in $\B(0,L+cn\epsilon)$ and with transition kernel given by \eqref{eq:kernel}. Let $\zeta$ be a random walk on the sphere $\S_R$ started from a uniform point in $\C(L) = \pi^{-1}(\B(0,L+cn\epsilon))$ and with transition kernel given by \eqref{eq:sphkern}. Then there exists a coupling of $z$ and $\pi(\zeta)$ such that for all $n$ we have that
\[
\P(\forall k=0, \ldots, n, \pi(\zeta(k)) = z(k)) \to 1 \text{ as } R \to \infty.
\]
\end{lemma}

\begin{proof}[\textbf{Proof}]
Since $\zeta(0)$ is uniformly distributed on $\C(L)$, we have that
$\pi(\zeta(0))$ has density function given by
\[
f_R^0(x) = \frac{\1(x \in \B(0,L+cn\epsilon))}{\cos(\alpha(x)) \mu(\C(L))},
\]
where $\alpha(x)$ is the angle between $\pi^{-1}(x)$ and $(0,\ldots,0,-R)$.
Also $z(0)$ has density function given by
\[
f^0(x) = \frac{\1(x \in \B(0,L+cn\epsilon))}{\vol(\B(0,L+cn\epsilon))}.
\]
From Lemma~\ref{lem:volume} we get that $f_R(x) \to f(x)$ as $R \to \infty$ uniformly over all $x$, and hence the maximal coupling will succeed with probability tending to $1$ as $R \to \infty$.
So, for $\pi(\zeta(0))$ and $z(0)$ we have used the maximal coupling. Given $\pi(\zeta(0))$ and $z(0)$, then
the density function of $\pi(\zeta(1))$ is given by
\[
f_R^1(x) = \frac{\1(x \in \pi(\C(\zeta(0),\epsilon)))}{\cos(\alpha(x)) \mu(\C(\zeta(0),\epsilon))}
\]
and $\xi(1)$ has density function given by
\[
f^1(x) = \frac{\1(x \in \B(\xi(0),\epsilon))}{\vol(\B(\xi(0),\epsilon))}.
\]
Thus if the coupling of the starting points has succeeded, then we can use the maximal coupling to couple the first steps and continuing this way we can couple the first $n$ steps. Hence, we get that the probability that the coupling succeeds tends to $1$ as $R \to \infty$.
\end{proof}

\begin{lemma}\label{lem:convol}
Let $(A_i)_{i=1,\ldots,n}$ be measurable sets in $\R^d$ such that $\vol(\partial A_i) = 0$, for all $i$. Then the function defined by
\[
x=(x_1,\ldots,x_n) \mapsto \vol(\cup_{i=1}^{n} (x_i + A_i))
\]
is continuous.
\end{lemma}

\begin{proof}[\textbf{Proof}]
Let $A^\delta$ stand for the $\delta$-enlargement of the set $A$, i.e. $A^\delta = A + \B(0,\delta)$. If $\vol(\partial A) = 0$, then it is easy to see that
\begin{align}\label{eq:delta}
\vol(A^{\delta}) \to \vol(A) \quad \text{ as } \delta \to 0.
\end{align}
Let $x=(x_1,\ldots,x_n)$ be such that $\|x\|< \delta$. We then have that
\[
\vol(\left(\cup_{i=1}^{n} (x_i + A_i)\right)\triangle \left(\cup_{i=1}^{n}A_i\right))
\leq \vol(\cup_{i=1}^{n}\left((x_i + A_i)\triangle A_i \right))
\leq \sum_{i=1}^{n} \vol((x_i+A_i)\triangle A_i).
\]
But $\vol((x_i+A_i)\triangle A_i) = 2 (\vol((x_i+A_i)\cup A_i) - \vol(A_i))$ and $\vol((x_i + A_i)\cup A_i) \leq \vol(A_i^{\delta})$, and hence by \eqref{eq:delta} we get the desired convergence.
\end{proof}

\section{Concluding Remarks and Questions}\label{sec:conclusion}
\begin{enumerate}
\item
We recall here the detection problem as discussed in \cite{PSSS10}.
Let $\Pi=\{X_i\}$ be a Poisson point process in $\R^d$ of intensity $\lambda$. We now let each point $X_i$ of the Poisson process move according to an independent standard Brownian motion $(\xi_i(s))_{s \geq 0}$. Let $u$ be another particle originally placed at the origin and which is moving according to a deterministic function $f$. We define the detection time of $u$ analogously to \eqref{eq:deftdet} via
\[
T_{\mathrm{det}}^f = \inf\{s \geq 0: \exists i \text{ s.t. } X_i + \xi_i(s) \in \B(f(s),r)\}.
\]
Then from \cite[Lemma~3.1]{PSSS10} we have that
\[
\P(T_{\mathrm{det}}^f> t) = \exp \left(-\lambda \E[\vol\left(\cup_{s \leq t}\B(\xi(s) + f(s),r) \right)] \right).
\]
In terms of the detection probabilities, Theorem~\ref{thm:wienersausage} then gives that
\[
\P(T_{\mathrm{det}}^f>t) \leq \P(T_{\mathrm{det}}^0>t),
\]
where $T_{\mathrm{det}}^0$ stands for the detection time when $u$ does not move at all. This means that the best strategy for $u$ to stay undetected for long time is to stay put.
This is an instance of the Pascal principle, which is discussed in \cite{DGRS} and \cite{MOBC} for a similar model in the discrete lattice.

\item
Let $(D_s)$ be a collection of open sets as in Theorem~\ref{thm:wienersausage}. We showed that
\[
\E[\vol(\cup_{s \leq t}(\xi(s) + D_s))] \geq \E[\vol(\cup_{s \leq t}(\xi(s) + B_s))],
\]
where the sets $(B_s)$ are balls as defined in Theorem~\ref{thm:wienersausage}. Does the stochastic domination inequality
\[
\P(\vol(\cup_{s \leq t}(\xi(s) + D_s)) \geq \alpha) \geq \P(\vol(\cup_{s \leq t}(\xi(s) + B_s)) \geq \alpha) \quad \forall \alpha
\]
also hold?

\item
We fix $D$ an open set in $\R^d$ and a measurable function $f$. We consider the Wiener sausage with drift $f$ determined by $D$, namely
$W^{D}_f(t) = \cup_{s \leq t} (D + \xi(s) + f(s))$. Is it true that the expected volume of $W^D_f$ is minimized when $f=0$?

\item
Let $f:\R_+ \to \R^d$ be a measurable function. Consider the convex hull of $C_f(t)=(\xi+f)[0,t]$. Is the expected volume of $C_f(t)$ minimized when $f=0$?

\end{enumerate}

\noindent{\Large\bf{Acknowledgements}}

We are grateful to Jason Miller, Alistair Sinclair and Alexandre Stauffer for useful discussions.

\bibliographystyle{plain}
\bibliography{biblio}

\end{document}